\documentclass[12pt]{amsart}
\calclayout

\usepackage{amsmath}
\usepackage{amssymb}
\usepackage{amsfonts}
\usepackage{amsthm}
\usepackage{mathrsfs}
\usepackage[all]{xy}
\usepackage{enumerate}
\usepackage{graphicx}
\usepackage{romannum}
\usepackage{xcolor}

\usepackage{physics}
\usepackage{relsize}
\usepackage{hyperref}

\usepackage[autostyle]{csquotes}

\def\CC{\mathbb{C}} 
 
\def\O{\Omega} 
\def\Levi{\mathscr{L}} 
\def\Lie{\mathcal{L}}  
\def\Null{\mathcal{N}}    

\newtheorem{thm}{Theorem}[section]
\newtheorem{lem}[thm]{Lemma}
\newtheorem{prop}[thm]{Proposition}

\theoremstyle{definition}
\newtheorem{defn}[thm]{Definition}

\newtheorem{rmk}[thm]{Remark}

\numberwithin{equation}{section}

\def\BigRoman{\uppercase\expandafter{\romannumeral\number\count 255 }}
\def\Romannumeral{\afterassignment\BigRoman\count255=}

\begin{document}
	
	\renewcommand{\thepage}{\arabic{page}}
	
	\title{CR-invariance of the Steinness index}         
	\author{Jihun Yum}        
	\address{Republic of Korea}
	\email{jihun0224@gmail.com}
	\thanks{}

	\begin{abstract}
		We characterize the Diederich-Fornaess index and the Steinness index in terms of a special 1-form, which we call D'Angelo 1-form. We then prove that the Diederich-Fornaess and Steinness indices are invariant under CR-diffeomorphisms by showing CR-invariance of D'Angelo 1-forms.
	\end{abstract}
	
	\maketitle

	\section{\bf Introduction} \label{section introduction}
	
	Let $\O \subset \CC^n$$(n \ge 2)$ be a bounded domain with $C^1$-smooth boundary. 
	A $C^1$-smooth function $\rho$ defined on a neighborhood $V$ of $\overline{\O}$ is called a (global) {\it defining function} of $\O$ if $\O = \{ z \in V : \rho(z)<0 \}$ and $d\rho(z) \neq 0$ for all $z \in \partial \O$.
	The {\it Diederich-Forn{\ae}ss exponent of $\rho$} is defined by 
	$$\eta_{\rho} := \sup\{ \eta \in (0,1) : -(-\rho)^{\eta} \text{ is strictly plurisubharmonic on } \O \}.$$
	If there is no such $\eta$, then we define $\eta_{\rho} = 0$. 
	{\it The Diederich-Forn{\ae}ss index of $\O$} is defined by 
	$$DF(\O) := \sup \text{ } \eta_{\rho},$$ 
	where the supremum is taken over all defining functions $\rho$. We say that the Diederich-Forn{\ae}ss index of $\O$ exists if $DF(\O) \in (0, 1]$. If $DF(\O)$ exists, then there exists a bounded strictly plurisubharmonic exhaustion function on $\O$. In other words, $\O$ becomes a hyperconvex domain. In 1977, Diederich and Forn{\ae}ss (\cite{DieForn1977}) proved that for a bounded pseudoconvex domain $\O$, $C^2$-smoothness of $\partial \O$ implies the existence of $DF(\O)$.

	In \cite{Yum2019}, the author introduced the concept of Steinness index which is an analogue of the Diederich-Fornaess index but related to Stein neighborhood bases.
	The {\it Steinness exponent of $\rho$} is defined by 
	\begin{align*}
		\widetilde{\eta}_{\rho} := \inf\{ \widetilde{\eta} > 1 : \rho^{\widetilde{\eta}} \text{ is strictly plurisubharmonic on } \overline{\Omega}^{\complement} \cap U  &\\
		\text{ for some neighborhood } U \text{ of } & \partial \O  \},
	\end{align*}
	where $\overline{\O}^{\complement}:= \CC^n \setminus \overline{\O}$.
	If there is no such $\widetilde{\eta}$, then we define $\widetilde{\eta}_{\rho} = \infty$. 
	The {\it Steinness index of $\Omega$} is defined by 
	$$S(\Omega) := \inf \text{ } \widetilde{\eta}_{\rho},$$ 
	where the infimum is taken over all defining functions $\rho$.  We say that the Steinness index of $\O$ exists if $S(\O) \in [1,\infty)$.
	$\overline{\O}$ is said to have a {\it Stein neighborhood basis} if for any neighborhood $V_1$ of $\overline{\O}$, there exists a pseudoconvex domain $V_2$ such that $\overline{\O} \subset V_2 \subset V_1$.
	If $S(\O)$ exists, then there exist a defining function $\rho$ and $\eta_2 \in (1,\infty)$ such that $\rho^{\eta_2}$ is strictly plurisubharmonic on $\overline{\O}^{\complement} \cap U$. Thus $\overline{\O}$ has a Stein neighborhood basis. In contrast to the Diederich-Forn{\ae}ss index, the smoothness of the boundary does not imply the existence of $S(\O)$; worm domains provide an example (\cite{DieForn1977-2}). The author (\cite{Yum2019}) showed that the existence of Steinness index is equivalent to that of a strong Stein neighborhood basis. Moreover, the author found the following relation between two indices on worm domains $\O_{\beta}$;
	$$ \frac{1}{DF(\O_{\beta})} + \frac{1}{S(\O_{\beta})} = 2. $$

	\vspace{3mm}

	The purpose of this paper is to investigate whether those two indices are invariant under CR-diffeomorphisms.  
	The author (\cite{Yum2017}) proved that for relatively compact domains $\O_1$ and $\O_2$ in Stein manifolds with $C^1$-smooth boundary, if there exists CR-diffeomorphism $f : \partial \O_1 \rightarrow \partial \O_2$, then $DF(\O_1) = DF(\O_2)$. The idea of the proof is the following: we first extend $f$ to a diffeomorphism $F : U_1 \supset \overline{\O}_1 \rightarrow U_2 \supset \overline{\O}_2$ such that $F|_{\O_1} : \O_1 \rightarrow \O_2$ is a biholomorphism and $F|_{\partial \O_1} = f$. Then for $0 < \eta < 1$ and a defining function $\rho$ of $\O_2$, if $-(-\rho)^{\eta}$ is strictly plurisubharmonic on $\O_2$, then $-(-(\rho \circ F))^{\eta}$ is also strictly plurisubharmonic on $\O_1$.

	On the other hand, this approach for proving CR-invariance of the Steinness index does not work, because we can not extend CR-diffeomorphisms holomorphically outside the domain near the boundary. 
	Instead, we characterize the Steinness index as well as the Diederich-Fornaess index on the set of weakly pseudoconvex boundary points in terms of a special 1-form, which we call D'Angelo 1-form (Definition \ref{def D'Angelo 1-form}).

	\begin{thm} \label{thm main DF,S D'Angelo 1-form}
	 	Let $\Omega \subset\subset \CC^n$ be a pseudoconvex domain with $C^k (k \ge 3)$-smooth boundary. Let $\rho$ be a defining function of $\O$. Let $\alpha_{\rho}$ be a D'Angelo 1-form of $\rho$. Denote by $\omega_{\rho} := \pi_{1,0} \alpha_{\rho}$ the projection of $\alpha_{\rho}$ onto its $(1,0)$-component. Then
	 	\begin{align*}
	 		DF(\O) &= \sup_{\rho} \left\{ 0 < \eta_1 < 1 : \left( \frac{\eta_1}{1 - \eta_1}(\omega_{\rho} \wedge \overline{\omega}_{\rho}) - \overline{\partial} \omega_{\rho}  \right) (L, \overline{L}) \le 0 \phantom{aa} \forall p \in \Sigma, \phantom{a} \forall L \in \Null_p \right\}, \\
	 		S(\O) &= \inf_{\rho} \left\{ \eta_2 > 1 : \left( \frac{\eta_2}{\eta_2 - 1}(\omega_{\rho} \wedge \overline{\omega}_{\rho}) + \overline{\partial} \omega_{\rho}  \right) (L, \overline{L}) \le 0 \phantom{aa} \forall p \in \Sigma, \phantom{a} \forall L \in \Null_p \right\},
	 	\end{align*}
	 	where $\Sigma$ is the set of weakly pseudoconvex boundary points and $\Null_p$ is the null-space of Levi-form at $p$.
	 \end{thm}

	 Then we prove that the two indices are invariant under CR-diffeomorphism by showing that D'Angelo 1-forms are invariant under CR-diffeomorphisms (Proposition \ref{prop CR-invariance of D'Angelo 1 form}).

	 \begin{thm} \label{thm main CR-invariance of two indices}
	 	Let $\O_1$ and $\O_2$ be bounded domains in $\CC^n$ with $C^k (k \ge 3)$-smooth boundaries. If $\partial \O_1$ and $\partial \O_2$ are CR-equivalent then 
	 	$$ DF(\O_1) = DF(\O_2) \phantom{aaa} \text{and} \phantom{aaa} S(\O_1) = S(\O_2) $$
	 	hold. 
	 \end{thm}

	 \vspace{5mm}


	\section{\bf Preliminaries}
	
	We first fix the notation of this paper, unless otherwise mentioned.
	\begin{itemize}
		\item[$\bullet$] $\O$ : a bounded pseudoconvex domain with $C^k (k \ge 3)$-smooth boundary in $\CC^n$.
		\item[$\bullet$] $\rho$ : a defining function of $\O$.
		\item[$\bullet$] $\Sigma$ : the set of all weakly pseudoconvex points in $\partial \O$.
		\item[$\bullet$] $g$ : the standard Euclidean complex Hermitian metric in $\CC^n$.
		\item[$\bullet$] $\nabla$ : the Levi-Civita connection of $g$.
		\item[$\bullet$] $\nabla \rho$ : the real gradient of $\rho$.
		\item[$\bullet$] $U$ : a tubular neighborhood of $\partial \O$.
		\item[$\bullet$] $\Levi_{\rho}$ : the Levi-form of $\rho$.
		\item[$\bullet$] $\Null_p$ : the null-space of the Levi-form at $p \in \partial \O$.
	\end{itemize}
	Define
	$$
	N_{\rho} := \frac{1}{\sqrt{ \sum^n_{j=1} |\frac{\partial \rho}{\partial z_j}|^2}} \sum^n_{j=1} \frac{\partial \rho}{\partial \bar{z}_j} \frac{\partial}{\partial z_j}, \phantom{aaa} 
	L_{n, \rho} := \frac{1}{\sum^n_{j=1} |\frac{\partial \rho}{\partial z_j}|^2} \sum^n_{j=1} \frac{\partial \rho}{\partial \bar{z}_j} \frac{\partial}{\partial z_j}.
	$$
	We denote $N_{\rho}$ and $L_{n,\rho}$ by $N$ and $L_n$, respectively, if there is no ambiguity.
	Note that $N$ and $L_n$ depend on a defining function $\rho$, but $N$ is independent of $\rho$ on $\partial \O$. Also, note that $\Re{N_{\rho}}$ and $\Re{L_n}$ are real normal vector fields on $\partial \O$, and 
	$$ \norm{N} = \sqrt{g(N,N)} = \frac{1}{\sqrt{2}}, \phantom{aaa}    L_n \rho = 1, \phantom{aaa}   L_n = \frac{1}{\norm{\partial \rho}}N = \frac{2}{\norm{\nabla \rho}}N. $$
	Let $J$ be the complex structure of $\CC^n$ and $T_p (\partial \O)$ be the real tangent space of $\partial \O$ at $p \in \partial \O$. Let $T_p^c(\partial \O) := J(T_p (\partial \O)) \cap T_p (\partial \O)$. Then the complexified tangent space of $T_p^c(\partial \O)$, $\CC T_p^c(\partial \O) := \CC \otimes T_p^c(\partial \O)$, can be decomposed into the holomorphic tangent space $T^{1,0}_p(\partial \O)$ and the anti-holomorphic tangent space $T^{0,1}_p(\partial \O)$. We call $X$ a $(1,0)$ tangent vector if $X \in T^{1,0}_p(\partial \O)$.
	
	Let $X,Y,Z$ be complex vector fields in $\CC^n$. 
	A direct calculation implies the following properties.
	$$ \Levi_{\rho}(X,Y) = g(\nabla_X \nabla \rho, Y) = X(\overline{Y} \rho) - (\nabla_X \overline{Y}) \rho,  $$
	\begin{align*}
	Z g(X,Y) = g(\nabla_Z X, Y) + g(X, \nabla_{\overline{Z}}Y), \phantom{aa}
	Z \rho = g(\nabla \rho, \overline{Z}),
	\end{align*}
	$$ \nabla_X\nabla_Y - \nabla_Y\nabla_X - \nabla_{[X,Y]} = 0. $$
	
	Let $d = \partial + \overline{\partial}$ be the exterior derivative and $d^c = i(\overline{\partial} - \partial)$. 
	
	\begin{lem} \label{orthgonal}
		Let $\Omega \subset\subset \CC^n$ be a pseudoconvex domain with $C^k (k \ge 2)$-smooth boundary, and $\rho$ be a defining function of $\Omega$. Suppose that $\Levi_{\rho}(L,L)(p) = 0$ for $p \in \partial \O$ and $L \in T^{1,0}_p(\partial \O)$. Then $\Levi_{\rho}(L,T)(p) = 0$ for all $T \in T^{1,0}_p(\partial \O)$.
	\end{lem}
	\begin{proof}
		Since $\Omega$ is pseudoconvex, the Levi-form of a defining function $\rho$ satisfies the Cauchy-Schwarz inequality on $T^{1,0}_p(\partial \O)$. Therefore, for all $T \in T^{1,0}_p(\partial \O)$, 
		$$ |\Levi_{\rho}(L,T)(p)|^2 \le |\Levi_{\rho}(L,L)(p)| |\Levi_{\rho}(T,T)(p)| = 0, $$ which completes the proof. 
		
	\end{proof}

	 \vspace{5mm}


	\section{\bf D'Angelo 1-form} \label{section D'Angelo 1-form}

	We first define a real 1-form $\alpha$ which was first introduced by D'Angelo (\cite{D'Angelo1980}, \cite{D'Angelo1987}).
	We call it {\it D'Angelo 1-form}. For a defining function $\rho$ of $\O$, define 
	 $$ \eta_{\rho} = \frac{1}{2}\left( \partial \rho - \overline{\partial} \rho \right)
	 \phantom{aa} \text{and} \phantom{aa} 
	 T_{\rho} = L_{n, \rho} - \overline{L}_{n, \rho}. $$
	We denote $\eta_{\rho}$ and $T_{\rho}$ by $\eta$ and $T$, respectively, if there is no ambiguity.
	Then $\eta$ is a purely imaginary, non-vanishing 1-form on $\partial \O$ that annihilates $T^{1,0}(\partial \O) \oplus T^{0,1}(\partial \O)$, and $T$ is a purely imaginary tangential vector field orthogonal to $T^{1,0}(\partial \O) \oplus T^{0,1}(\partial \O)$ such that $\eta(T) \equiv 1$. 
	
	\begin{defn} \label{def D'Angelo 1-form}
		A {\it D'Angelo 1-form} $\alpha_{\rho}$ on $\partial \O$ is defined by 
		$$ \alpha_{\rho} := - \Lie_T \eta ,$$
		where $\Lie_T$ is the Lie derivative in the direction of $T$.
	\end{defn}

	\begin{rmk}
		Note that the D'Angelo 1-form is defined from a defining function. On the other hand, we can also define the D'Angelo 1-form without using a defining function. 
		Let $\eta$ is a purely imaginary, non-vanishing 1-form on $\partial \O$ that annihilates $T^{1,0}(\partial \O) \oplus T^{0,1}(\partial \O)$, and $T_{\eta}$ is the (uniqe) purely imaginary tangential vector field orthogonal to $T^{1,0}(\partial \O) \oplus T^{0,1}(\partial \O)$ such that $\eta(T) \equiv 1$. Then the D'Angelo 1-form $\alpha_{\eta}$ of $\eta$ is defined by
		$$ \alpha_{\eta} := - \Lie_{T_{\eta}} \eta ,$$
		where $\Lie_{T_{\eta}}$ is the Lie derivative in the direction of $T_{\eta}$. We claim that those two definitions are equivalent. In order words, we prove that for any $\eta$ defined above there exists a defining function $\rho$ such that $\alpha_{\rho} = \alpha_{\eta}$.
		Fix a defining function $r$ of $\O$. Then since $\eta$ is non-vanishig, there exists a smooth function $\varphi$ on $\partial \O$ such that $\eta = e^{\varphi} \eta_r$ or $\eta = - e^{\varphi} \eta_r$. First, suppose that $\eta = e^{\varphi} \eta_r$. Then $T_{\eta} = e^{-\varphi} T_r$. 
		Let $\rho := e^{\varphi} r$. Then $\eta_{\rho} = e^{\varphi} \eta_r = \eta$ and $T_{\rho} = e^{-\varphi} T_r = T_{\eta}$. Therefore, 
		$$\alpha_{\rho} = - \Lie_{T_{\rho}} \eta_{\rho} = - \Lie_{T_{\eta}} \eta = \alpha_{\eta}. $$
		If $\eta = -e^{\varphi} \eta_r$, then by the same argument above, 
		$\alpha_{\rho} = - \Lie_{T_{\rho}} \eta_{\rho} = - \Lie_{(-T_{\eta})} (-\eta) = \alpha_{\eta}$. In this paper, we use Definition \ref{def D'Angelo 1-form} for our purpose. 
	\end{rmk}

	We denote $\alpha_{\rho}$ by $\alpha$ if there is no ambiguity. Note that since $\eta$ and $T$ are purely imaginary, $\alpha$ is a real 1-form. Now we give known properties of $\alpha$ which we will use later. For more information about the D'Angelo 1-form, we refer the reader to \cite{Straube2010}.

	\begin{prop} [\cite{BoasStraube1993}] \label{prop d(alpha) vanishes}
		Let $\alpha$ be a D'Angelo 1-form on $\partial \O$. Then 
		$$ (d\alpha)_p (X,Y) = 0 $$
		for all $p \in \partial \O$ and $X, Y \in \Null_p \oplus \overline{\Null}_p $.
	\end{prop}

	\begin{lem}[\cite{BoasStraube1993}] \label{lem family of good vector fields 1}
		Let $L_1, L$ be smooth $(1,0)$ tangent vector fields on $\partial \O$. Then
		$$ \partial \rho ( [L_n, \overline{L}] ) = \Levi_{\rho}(L_n,L) ,$$
		and
		$$ \partial \rho ( [L_1, \overline{L}] ) = \Levi_{\rho}(L_1,L) $$
		on $\partial \O$.
	\end{lem}
	\begin{proof}
		Let $L_1 = \sum_{j=1}^{n}a_j \frac{\partial}{\partial z_j}$ and $L = \sum_{k=1}^{n}b_k \frac{\partial}{\partial z_k}$.
		Since $L_n \rho = 1$, $L_1 \rho = 0$,
		$$ \sum^n_{j=1} \left( \frac{1}{\norm{\partial \rho}^2} \frac{\partial \rho}{\partial \overline{z}_j} \right) \left( \frac{\partial \rho}{\partial z_j} \right) = 1 ,$$
		$$ \sum^n_{j=1} a_j \frac{\partial \rho}{\partial z_j} = 0 $$
		imply
		$$ \sum^n_{j=1} \overline{L} \left( \frac{1}{\norm{\partial \rho}^2} \frac{\partial \rho}{\partial \overline{z}_j} \right) \left( \frac{\partial \rho}{\partial z_j} \right) 
		+ \sum^n_{j=1} \left( \frac{1}{\norm{\partial \rho}^2} \frac{\partial \rho}{\partial \overline{z}_j} \right) \overline{L} \left( \frac{\partial \rho}{\partial z_j} \right) = 0 ,$$
		$$ \sum_{j=1}^{n} \overline{L}(a_j) \frac{\partial \rho}{\partial z_j} + \sum_{j=1}^{n} a_j \overline{L} \left( \frac{\partial \rho}{\partial z_j} \right)= 0 ,$$
		respectively.
		Therefore,
		\begin{align*}
		\partial \rho \left( [L_n , \overline{L}] \right) 
		&= - \sum^n_{j=1} \overline{L} \left( \frac{1}{\norm{\partial \rho}^2} \frac{\partial \rho}{\partial \overline{z}_j} \right) \frac{\partial \rho}{\partial z_j}
		= \sum^n_{j=1} \frac{1}{\norm{\partial \rho}^2} \frac{\partial \rho}{\partial \overline{z}_j} \overline{L} \left( \frac{\partial \rho}{\partial z_j} \right) \\
		&= \frac{1}{\norm{\partial \rho}^2} \sum^n_{j=1,k=1} \frac{\partial^2 \rho}{\partial z_j \partial \overline{z}_k} \frac{\partial \rho}{\partial \overline{z}_j} \overline{b}_k
		= \Levi_{\rho}(L_n, L),
		\end{align*}
		and
		\begin{align*}
		\partial \rho \left( [L_1, \overline{L}] \right)
		&= - \sum^n_{j=1} \overline{L} (a_j) \frac{\partial \rho}{\partial z_j}
		= \sum^n_{j=1} a_j \overline{L} \left( \frac{\partial \rho}{\partial z_j} \right) \\
		&= \sum^n_{j=1,k=1} \frac{\partial^2 \rho}{\partial z_j \partial \overline{z}_k} a_j \overline{b}_k
		= \Levi_{\rho}(L_1, L).
		\end{align*}
	\end{proof}

	\begin{lem}[\cite{BoasStraube1993}] \label{lem alpha(L)}
		Let $\alpha$ be a D'Angelo 1-form on $\partial \O$. Then
		$$ \alpha(\overline{L}) = \partial \rho \left( [L_n , \overline{L}] \right) = 2\frac{\Levi_{\rho}(N,L)}{\norm{\nabla \rho}} $$
		on $\partial \O$ for a $(1,0)$ vector field $L$ on $\partial \O$.
	\end{lem}
	\begin{proof}
		Lemma \ref{lem family of good vector fields 1} implies that
		\begin{align*}
		\alpha(\overline{L}) 
		&= - (\Lie_T \eta)(\overline{L}) 
		= - T(\eta(\overline{L})) + \eta([ T,\overline{L} ]) 
		= \eta([ T,\overline{L} ]) 
		= \frac{1}{2}\left( \partial \rho - \overline{\partial} \rho \right) ( [T,\overline{L}] )	\\	
		&= \partial \rho \left( [L_n - \overline{L}_n, \overline{L}] \right)
		= \partial \rho \left( [L_n , \overline{L}] \right)
		= \Levi_{\rho}(L_n,L)
		= 2 \frac{\Levi_{\rho}(N,L)}{\norm{\nabla \rho}}.
		\end{align*}
		Here, we used $ \left( \partial \rho + \overline{\partial}\rho \right) \left( [T, \overline{L}] \right) = d\rho \left( [T, \overline{L}] \right) = 0$ (because $[T,\overline{L}]$ is tangential), and $\partial \rho \left([\overline{L}_n, \overline{L}] \right) = 0$ (because $[\overline{L}_n, \overline{L}]$ is of type $(0,1)$).
	\end{proof}


	\vspace{5mm}
	
	\section{\bf Characterization of two indices in terms of D'Angelo 1-form}

	The author (\cite{Yum2019}), exploiting the idea of \cite{Liu2017}, characterized the Steinness index by way of a differential inequality on the set of weakly pseudoconvex boundary points (in fact, on the set of infinite type boundary points in the sense of D'Angelo). Also, we may induce a similar description for the Diederich-Fornaess index by the same argument. The two descriptions are the following. \\
	
	Throughout this section, let $L$ be an arbitrary (1,0) vector field defined on a neighborhood $U \supset \partial \O$ with $L \rho = 0$ and define 
	$$ \Sigma_L := \{ p \in \partial \O : \Levi_{\rho}(L,L)(p) = 0 \} .$$
	
	\begin{thm} \label{thm DF Yum2019}
		$$ DF(\O) = \sup_{\rho} \left\{ 0 < \eta_1 < 1 :   
		\frac{1}{1-\eta_1} \frac{|\Levi_{\rho}(L,N)|^2}{\lVert \nabla \rho \rVert^2} + \frac{1}{2} \frac{N\Levi_{\rho}(L,L)}{\lVert \nabla \rho \rVert} \le 0 
		\text{ on } \Sigma_L \phantom{a} \forall L
		\right\},
		$$	
		where the supremum is taken over all smooth defining functions $\rho$.
	\end{thm}

	\begin{thm} [Yum \cite{Yum2019}] \label{thm S Yum2019}
		$$ S(\O) = \inf_{\rho} \left\{ \eta_2 > 1 :   
		\frac{1}{\eta_2 - 1} \frac{|\Levi_{\rho}(L,N)|^2}{\lVert \nabla \rho \rVert^2} - \frac{1}{2} \frac{N\Levi_{\rho}(L,L)}{\lVert \nabla \rho \rVert} \le 0 
		\text{ on } \Sigma_L \phantom{a} \forall L
		\right\},
		$$	
		where the infimum is taken over all smooth defining functions $\rho$.
	\end{thm}

	\begin{rmk}
		We give here an explaination about the difference of Theorem \ref{thm DF Yum2019} and the formula of Liu (Thoerem 2.9 in \cite{Liu2017}). Liu dealt with a defining function of the form $r e^{\psi}$, where $r$ is a defining function and $\psi$ is a smooth function near the boundary. If we apply the same argument as Liu for a defining function $\rho$, then one can conclude Theorem \ref{thm DF Yum2019}. If we replace $\rho$ by $r e^{\psi}$ in Theorem \ref{thm DF Yum2019} (using Lemma 2.2 in \cite{Yum2019}), one gets a formula which is equivalent to the formula of Liu.

	\end{rmk}
	
	The two theorems above imply that the Diederich-Fornaess index and Steinness index are completely determined by the two values $\frac{|\Levi_{\rho}(L,N)|^2}{\lVert \nabla \rho \rVert^2}$ and $\frac{N\Levi_{\rho}(L,L)}{\lVert \nabla \rho \rVert}$ on the set of weakly pseudoconvex boundary points. From Lemma \ref{lem alpha(L)}, we know that 
	$$ \alpha(\overline{L}) = 2\frac{\Levi_{\rho}(N,L)}{\norm{\nabla \rho}} $$
	for a D'Angelo 1-form $\alpha$.
	We now describe the second value in terms of $\alpha$.

	\begin{lem} \label{lem In terms of the D'Angelo one form 2}
		Let $\alpha$ be a smooth real 1-form on $\CC^n$. Let $X \in T^{1,0}(\CC^n)$ and $\overline{Y} \in T^{0,1}(\CC^n)$.
		Then the following equations hold.
		\begin{align*}
		(d \alpha)(X, \overline{Y}) &= X( \alpha(\overline{Y}) ) - \overline{Y}( \alpha(X) ) - \alpha(\nabla_X \overline{Y}) + \alpha(\nabla_{\overline{Y}} X), \\
		(d^c \alpha)(X, \overline{Y}) &= i \left[ -X( \alpha(\overline{Y}) ) - \overline{Y}( \alpha(X) ) + \alpha(\nabla_X \overline{Y}) + \alpha(\nabla_{\overline{Y}} X) \right].
		\end{align*}
	\end{lem}
	
	\begin{proof}
		The first equation follows by the definition of exterior derivative : 
		\begin{align*}
		(d \alpha)(X, \overline{Y}) &= X( \alpha(\overline{Y}) ) - \overline{Y}( \alpha(X) ) - \alpha([X,\overline{Y}]) \\
		&= X( \alpha(\overline{Y}) ) - \overline{Y}( \alpha(X) ) - \alpha(\nabla_X \overline{Y}) + \alpha(\nabla_{\overline{Y}} X).
		\end{align*}
		Since $\alpha$ is real, we may write 
		$$ \alpha = \sum_{l=1}^{n} ( a_l dz_l + \overline{a}_l d\overline{z}_l ), \phantom{aaa} 
		X = \sum_{j=1}^{n} x_j \frac{\partial}{\partial z_j}, \phantom{aaa} 
		Y = \sum_{k=1}^{n} \overline{y}_k \frac{\partial}{\partial \overline{z}_k} $$ 
		for some $a_j \in \CC$. Then
		
		\begin{align*}
		& d^c \alpha = i(\overline{\partial} - \partial) \alpha \\
		= & i \sum_{l,m=1}^{n} \left[ 
		\left(  \frac{\partial a_l}{\partial \overline{z}_m} d\overline{z}_m \wedge dz_l  + \frac{\partial \overline{a}_l}{\partial \overline{z}_m} d\overline{z}_m \wedge d\overline{z}_l  \right) -  
		\left(  \frac{\partial a_l}{\partial z_m} dz_m \wedge dz_l  + \frac{\partial \overline{a}_l}{\partial z_m} dz_m \wedge d\overline{z}_l  \right)          \right],
		\end{align*}
		\begin{align*}
		(d^c \alpha)(X, \overline{Y}) = i \sum_{j,k=1}^{n} \left[ 
		- \frac{\partial a_j}{\partial \overline{z}_k} x_j \overline{y}_k 
		- \frac{\partial \overline{a}_k}{\partial z_j} x_j \overline{y}_k  \right],
		\end{align*}
		\begin{align*}
		X(\alpha(\overline{Y})) &= X \left( \sum_{l=1}^{n} \overline{a}_l \overline{y}_l \right) = \sum_{j,k=1}^{n} x_j\frac{\partial \overline{a}_k}{\partial z_j} \overline{y}_k + \sum_{j,k=1}^{n} x_j \overline{a}_k \frac{\partial \overline{y}_k}{\partial z_j}, \\
		\overline{Y}(\alpha(X)) &= \overline{Y} \left( \sum_{l=1}^{n}  a_l x_l   \right) = \sum_{j,k=1}^{n}  \overline{y}_k \frac{\partial a_j}{\partial \overline{z}_k} x_j + \sum_{j,k=1}^{n} \overline{y}_k a_j \frac{\partial x_j}{\partial \overline{z}_k},
		\end{align*}
		\begin{align*}
		\nabla_X \overline{Y} = \sum_{j,k=1}^{n} x_j \frac{\partial \overline{y}_k}{\partial z_j} \frac{\partial}{\partial \overline{z}_k}, \phantom{aaa}
		\nabla_{\overline{Y}} X = \sum_{j,k=1}^{n} \overline{y}_k \frac{\partial x_j}{\partial \overline{z}_k} \frac{\partial}{\partial z_j},
		\end{align*}
		\begin{align*}
		\alpha( \nabla_X \overline{Y} ) = \sum_{j,k=1}^{n} \overline{a}_k x_j \frac{\partial \overline{y}_k}{\partial z_j}         , \phantom{aaa}
		\alpha( \nabla_{\overline{Y}} X ) = \sum_{j,k=1}^{n} a_j \overline{y}_k \frac{\partial x_j}{\partial \overline{z}_k},
		\end{align*}
		Therefore,
		\begin{align*}
		(d^c \alpha)(X, \overline{Y}) = i \left[ -X( \alpha(\overline{Y}) ) - \overline{Y}( \alpha(X) ) + \alpha(\nabla_X \overline{Y}) + \alpha(\nabla_{\overline{Y}} X)         \right].
		\end{align*}
	\end{proof}

	\begin{lem} \label{lem In terms of the D'Angelo one form 1}
		\begin{align} \label{4.0.1}
		\frac{|\Levi_{\rho}(L,N)|^2}{\norm{\nabla \rho}^2} + \frac{1}{2} \frac{N \Levi_{\rho}(L,L)}{\norm{\nabla \rho}} = \frac{1}{4} g(\nabla_L \nabla_{L_n} \nabla \rho, L)
		\end{align}
		holds at $p \in \Sigma_L$ for all $L$.
	\end{lem}
	
	\begin{proof}
		Fix $p \in \Sigma_L$. Then 
		\begin{align*}
		\frac{1}{2} \frac{N \Levi_{\rho}(L,L)}{\norm{\nabla \rho}} &= \frac{1}{4} L_n \Levi_{\rho}(L,L) = \frac{1}{4} L_n g(\nabla_L \nabla \rho, L) \\
		&= \frac{1}{4} g(\nabla_{L_n} \nabla_L \nabla \rho, L) + \frac{1}{4} g(\nabla_L \nabla \rho, \nabla_{\overline{L}_n} L) \\
		&= \frac{1}{4} g(\nabla_L \nabla_{L_n} \nabla \rho, L) + \frac{1}{4} g(\nabla_{[L_n, L]} \nabla \rho, L) + \frac{1}{4} g(\nabla_L \nabla \rho, \nabla_{\overline{L}_n} L)
		\end{align*}
		at $p$. If $L$ vanishes at $p$, then the last three terms are all zero and hence $N\Levi_{\rho}(L,L) = 0$. Thus the equation \ref{4.0.1} holds. Therefore, we assume that $L \neq 0$ at $p$.

		Let $\{\sqrt{2}T_1, \cdots, \sqrt{2}T_{n-2}, \sqrt{2}T_{n-1} \} $ be an orthonormal basis of $T^{1,0}_p(\partial \O)$ with $\sqrt{2}T_{n-1} = \frac{L}{\norm{L}}$. Then first,

		\begin{align*}
		& g(\nabla_{L} \nabla \rho, \nabla_{\overline{L}_n} L) \\
		= & g \left(  \nabla_{L} \nabla \rho, \sum_{j=1}^{n-1}g(\nabla_{\overline{L}_n} L, \sqrt{2}T_j)\sqrt{2}T_j  +  g(\nabla_{\overline{L}_n} L, \sqrt{2}N)\sqrt{2}N   \right) \\
		= & \sum_{j=1}^{n-1} 2 \overline{g(\nabla_{\overline{L}_n} L, T_j)} \Levi_{\rho}(L, T_j)  +  2 \overline{g(\nabla_{\overline{L}_n} L, N)} \Levi_{\rho}(L,N) \\
		= & 2 \overline{g(\nabla_{\overline{L}_n} L, N)} \Levi_{\rho}(L,N) \\
		= & 4 \overline{g(\nabla_{\overline{N}} L, N)} \frac{\Levi_{\rho}(L,N)}{\norm{\nabla \rho}}.
		\end{align*}
		Here, we used Lemma \ref{orthgonal} in the third equality above.
		\begin{align*}
		& \Levi_{\rho}(N,L) = N(\overline{L}\rho) - (\nabla_N \overline{L})\rho = - (\nabla_N \overline{L})\rho \\
		\Rightarrow \phantom{a} & - \Levi_{\rho}(L,N) = (\nabla_{\overline{N}} L)\rho = 2 g(\nabla_{\overline{N}} L, N)(N\rho) = g(\nabla_{\overline{N}} L, N) \lVert \nabla \rho \rVert \\
		\Rightarrow \phantom{a} & g(\nabla_{\overline{N}} L, N) = - \frac{\Levi_{\rho}(L,N)}{\lVert \nabla \rho \rVert}. 
		\end{align*}
		Therefore, 
		\begin{align*}
		g(\nabla_{L} \nabla \rho, \nabla_{\overline{L}_n} L) = 
		- 4 \frac{|\Levi_{\rho}(L,N)|^2}{\norm{\nabla \rho}^2}. \\
		\end{align*}
		Second, by the same argument as above,
		\begin{align*}
		g(\nabla_{[L_n,L]} \nabla \rho, L) 
		=  \overline{g(\nabla_{L} \nabla \rho, [L_n,L])} 
		=  2 g([L_n,L], N) \overline{\Levi_{\rho}(L, N)}.
		\end{align*}
		\begin{align*}
		g([L_n,L], N) &= g([L_n,L], N + \overline{N}) = \frac{1}{\norm{\nabla \rho}} g([L_n,L], \nabla \rho) = \frac{1}{\norm{\nabla \rho}} [L_n,L] \rho \\
		&= \frac{1}{\norm{\nabla \rho}} \left( L_n (L\rho) - L(L_n \rho) \right) = 0.
		\end{align*}
		The last equation follows from $L\rho = 0$ and $L_n \rho = 1$. Therefore, 
		\begin{align*}
		g(\nabla_{[L_n,L]} \nabla \rho, L) = 0.
		\end{align*}
		All together, we have
		\begin{align*} 
		\frac{|\Levi_{\rho}(L,N)|^2}{\norm{\nabla \rho}^2} + \frac{1}{2} \frac{N \Levi_{\rho}(L,L)}{\norm{\nabla \rho}} = \frac{1}{4} g(\nabla_L \nabla_{L_n} \nabla \rho, L).
		\end{align*}
		
	\end{proof}

	\begin{prop} \label{prop d^c(alpha)(L,L)}
		Let $\alpha$ be a D'Angelo 1-form. Then 
		\begin{align} 
		\frac{|\Levi_{\rho}(L,N)|^2}{\norm{\nabla \rho}^2} + \frac{1}{2} \frac{N \Levi_{\rho}(L,L)}{\norm{\nabla \rho}} = \frac{i}{8}(d^c \alpha)(L, \overline{L})
		\end{align}
		holds at $p \in \Sigma_L$ for all $L$.
	\end{prop}
	
	\begin{proof}
		By Lemma \ref{lem In terms of the D'Angelo one form 1}, 
		\begin{align*} 
		&\frac{|\Levi_{\rho}(L,N)|^2}{\norm{\nabla \rho}^2} + \frac{1}{2} \frac{N \Levi_{\rho}(L,L)}{\norm{\nabla \rho}} \\
		=& \frac{1}{4} g(\nabla_L \nabla_{L_n} \nabla \rho, L) \\
		=& \frac{1}{4} L g(\nabla_{L_n} \nabla \rho, L) - \frac{1}{4} g(\nabla_{L_n} \nabla \rho, \nabla_{\overline{L}} L ) \\
		=& \frac{1}{4} L \Levi_{\rho}(L_n, L) - \frac{1}{4} \Levi_{\rho}(L_n, \nabla_{\overline{L}} L)
		\end{align*}
		Since $\Levi_{\rho}(L,L) = 0$ at $p$,
		\begin{align} \label{equ 4.31}
			0 = \Levi_{\rho}(L,L) = \overline{L}( L\rho ) - (\nabla_{\overline{L}} L) \rho = - (\nabla_{\overline{L}} L) \rho
		\end{align}
		implies that $ (\nabla_{\overline{L}} L) \in T^{1,0}_p(\partial \O)$.
		Then by Lemma \ref{lem alpha(L)}, 
		\begin{align*}
		\alpha(\overline{L}) &= 2\frac{\Levi_{\rho}(N,L)}{\norm{\nabla \rho}} = \Levi_{\rho}(L_n, L), \\
		\alpha(\nabla_{L} \overline{L}) &= 2\frac{\Levi_{\rho}(N,\nabla_{\overline{L}} L)}{\norm{\nabla \rho}} = \Levi_{\rho}(L_n, \nabla_{\overline{L}} L).
		\end{align*}
		Note that we did not use $\Levi_{\rho}(L,L) = 0$ in the proof of Lemma \ref{lem alpha(L)}; we just used $L \rho=0$.
		Altogether, we have
		$$ \frac{|\Levi_{\rho}(L,N)|^2}{\norm{\nabla \rho}^2} + \frac{1}{2} \frac{N \Levi_{\rho}(L,L)}{\norm{\nabla \rho}} = \frac{1}{4} L(\alpha(\overline{L})) - \frac{1}{4} \alpha(\nabla_{L} \overline{L}). $$
		Now by Lemma \ref{lem In terms of the D'Angelo one form 2} and the above equation, 
		\begin{align} \label{equ 4.3}
		\frac{1}{8}(d\alpha + id^c \alpha)(L, \overline{L}) = \frac{1}{4} L( \alpha(\overline{L})) - \frac{1}{4} \alpha(\nabla_{L} \overline{L}) = \frac{|\Levi_{\rho}(L,N)|^2}{\norm{\nabla \rho}^2} + \frac{1}{2} \frac{N \Levi_{\rho}(L,L)}{\norm{\nabla \rho}}.
		\end{align}
		Since $d\alpha(L, \overline{L}) = 0$ by Proposition \ref{prop d(alpha) vanishes}, we have
		\begin{align*} 
		\frac{|\Levi_{\rho}(L,N)|^2}{\norm{\nabla \rho}^2} + \frac{1}{2} \frac{N \Levi_{\rho}(L,L)}{\norm{\nabla \rho}} = \frac{i}{8}(d^c \alpha)(L, \overline{L}).
		\end{align*}

	\end{proof}

	\begin{proof}[\bf Proof of Theorem \ref{thm main DF,S D'Angelo 1-form}]
		We prove it for the Steinness index case. Since $\alpha$ is a real 1-form, we may write it as $\alpha = \omega + \overline{\omega}$. Then
		\begin{align} \label{4.5}
			\frac{i}{2}d^c\alpha(L,\overline{L}) = \frac{1}{2}(d+id^c)\alpha(L,\overline{L})
			= \partial \alpha(L,\overline{L}) = \partial \overline{\omega}(L,\overline{L})
			= - \overline{\partial} \omega(L,\overline{L}).
		\end{align}
		Here, Proposition \ref{prop d(alpha) vanishes} is used in the first equality.
		Hence, from Lemma \ref{lem alpha(L)} and Proposition \ref{prop d^c(alpha)(L,L)},
		\begin{align*}
		&\frac{1}{\eta_2 - 1} \frac{|\Levi_{\rho}(L,N)|^2}{\lVert \nabla \rho \rVert^2} - \frac{1}{2} \frac{N\Levi_{\rho}(L,L)}{\lVert \nabla \rho \rVert} \le 0 \\
		\Leftrightarrow \phantom{aa} & \left( \frac{\eta_2}{\eta_2 - 1}(\omega \wedge \overline{\omega}) - \frac{i}{2} d^c \alpha \right) (L, \overline{L}) \le 0 \\
		\Leftrightarrow \phantom{aa} & \left( \frac{\eta_2}{\eta_2 - 1}(\omega \wedge \overline{\omega}) + \overline{\partial} \omega \right) (L, \overline{L}) \le 0. 
		\end{align*}
		Therefore, the theorem is proved by Theorem \ref{thm S Yum2019}. For the Diederich-Fornaess index case, it follows from the same argument as above.
		
	\end{proof}

	\begin{rmk}
		We give an alternative proof for Lemma 2.2 in \cite{Yum2019}.
		Let $\widetilde{\rho}$ and $\rho$ be $C^k (k \ge 3)$-smooth defining functions of $\O$. Then there exists a $C^{k-1}$-smooth real-valued function $\psi$ such that $\widetilde{\rho} = \rho e^{\psi}$. Let $\widetilde{\eta} = \eta_{\widetilde{\rho}}$, $\eta = \eta_{\rho}$, $\widetilde{T} = T_{\widetilde{\rho}}$ and $T = T_{\rho}$ as in Section \ref{section D'Angelo 1-form}.
		Let $\widetilde{\alpha} = -\Lie_{\widetilde{T}} \widetilde{\eta}$ and $\alpha = -\Lie_T \eta$.
		By a direct calculation, $\widetilde{\eta} = e^{\psi} \eta$ and $\widetilde{T} = e^{-\psi} T$.
		Then for $X \in T^{1,0}(\partial \O)$, as in the proof of Lemma \ref{lem alpha(L)},
		\begin{align*}
		\widetilde{\alpha}(X) = \widetilde{\eta}([\widetilde{T}, X]) = e^{\psi} \eta([e^{-\psi}T, X] ) = \eta([T,X]) + X\psi = \alpha(X) + d\psi(X).
		\end{align*}
		Since $\widetilde{\alpha}, \alpha, \psi$ are real, 
		$$ \widetilde{\alpha}(\overline{X}) = \alpha(\overline{X}) + d\psi(\overline{X}). $$ 
		Moreover, by the equation (\ref{equ 4.3}) and the above equation, for $L \in \Null_p$, $p \in \partial \O$,
		\begin{align*}
			id^c \widetilde{\alpha} (L, \overline{L}) 
			=& 2 L(\widetilde{\alpha}(\overline{L})) - 2 \widetilde{\alpha}(\nabla_{L}\overline{L}) \\
			=& 2 L \left( \alpha(\overline{L}) + \overline{L}\psi \right) - 2 \left( \alpha(\nabla_{L}\overline{L}) + (\nabla_{L}\overline{L})\psi \right) \\
			=& 2 L (\alpha(\overline{L})) - 2 \alpha(\nabla_{L}\overline{L}) +
			2 L(\overline{L}\psi) - 2 (\nabla_{L}\overline{L})\psi \\
			=& id^c \alpha(L,\overline{L}) - idd^c \psi (L,\overline{L}) .
		\end{align*}
		In the second equlity, we used $\nabla_{L}\overline{L} \in T^{0,1}(\partial \O)$ (see the equation (\ref{equ 4.31})).
	\end{rmk}
	

	\vspace{5mm}


	\section{\bf CR-invariance of two indices}
	
	In the previous section, we showed that the Diederich-Fornaess index and Steinness index are completely determined by D'Angelo 1-forms on the set of weakly pseudoconvex boundary points.
	In this section, we prove that those two indices are invariant under CR-diffeomorphisms by showing that D'Angelo 1-forms are invariant under CR-diffeomorphisms in some sense.

	\begin{defn}
		Let $\O_1$ and $\O_2$ be domains in $\CC^n$ $(n \ge 2)$ with $C^k (k \ge 1)$-smooth boundaries. A $C^k$-smooth function $f : \O_1 \rightarrow \CC$ is called a {\it CR-function} if
		$$ \overline{L} f = 0, $$
		for all $p \in \partial \O_1$ and $L \in T^{1,0}_p(\partial \O_1)$. A $C^k$-smooth function $F : \partial \O_1 \rightarrow \partial \O_2$ is called a {\it CR-map} if 
		$$ F(T^{1,0}_p(\partial \O_1)) \subset T^{1,0}_{f(p)}(\partial \O_2), $$
		for all $p \in \partial \O_1$.  A $C^k$-smooth function $F : \partial \O_1 \rightarrow \partial \O_2$ is called a {\it CR-diffeomorphism} if it is a CR-map and a diffeomorphism. 
	\end{defn}

	\begin{thm}[Boggess \cite{Boggess1991}] \label{thm CR-extension dbar f vanishes}
		Suppose $M$ is a $C^k$, $k \ge 2$, generic CR submanifold of $\CC^n$ with real dimension $2n-d$, $1 \le d \le n$. If $f$ is a $C^k$ CR-function on $M$, then there exists a $C^k$ function $F$ defined on $\CC^n$ such that $\overline{\partial} F$ vanishes on $M$ to order $k-1$ and $F|_M = f$.
	\end{thm}

	\begin{prop} \label{prop CR-invariance of D'Angelo 1 form}
		Let $\O_1$ and $\O_2$ be bounded pseudoconvex domains in $\CC^n$ with $C^k (k \ge 3)$-smooth boundaries, and $\rho$ be a defining function of $\O_2$. Suppose that there exists a $C^k$-smooth CR-diffeomorphism $f : \partial \O_1 \rightarrow \partial \O_2$. Then 
		\begin{align} \label{CR-invariance of D'Angelo 1 form}
		\alpha_{(\rho \circ f)}(\overline{L})(p) &= \alpha_{\rho}(\overline{f_{*}(L)})(f(p)), \\
		d^c \alpha_{(\rho \circ f)}(L, \overline{L})(p) &= d^c \alpha_{\rho}(f_{*}(L), \overline{f_{*}(L)})(f(p)), \nonumber
		\end{align}
		for all $p \in \partial \O_1$ and $L \in \Null_{p}$.
	\end{prop}
	
	\begin{proof}
		Let $\{z_j\}_{j=1}^n$ and $\{w_j\}_{j=1}^n$ be coordinates of $\O_1 \subset \CC^n$ and ${\O_2 \subset \CC^n}$, respectively.
		By Theorem \ref{thm CR-extension dbar f vanishes}, we may extend $f$ to $F$ such that $\overline{\partial} F$ vanishes on $\partial \O_1$ to order $k-1$ and $F|_M = f$. We will denote the extension $F$ again by $f$. Then,
		$$ \Levi_{(\rho \circ f)}(X, Y)(p) = \Levi_{\rho}(f_{*}(X), f_{*}(Y))(f(p)) $$
		for all $p \in \partial \O_1$ and $X,Y \in T^{1,0}_p(\partial \O_1)$.
		Denote
		$$ L_{n, \rho} := \frac{1}{\sum^n_{j=1} |\frac{\partial \rho}{\partial w_j}|^2} \sum^n_{j=1} \frac{\partial \rho}{\partial \overline{w}_j} \frac{\partial}{\partial w_j}. $$

		\noindent	
		{\bf Claim.} $f_{*}(L_{n,(\rho \circ f)}) = L_{n,\rho} + \widetilde{L}$ for some $\widetilde{L} \in T^{1,0}_{f(p)}(\partial \O_2)$.

		Let $f = (f_1, \cdots, f_n)$. Then for $p \in \partial \O_1$ and $q := f(p) \in \partial \O_2$,
		$$ \left. \frac{\partial (\rho \circ f)}{\partial \overline{z}_j} \right|_p 
		= \sum_{k=1}^n \left.\frac{\partial \rho}{\partial \overline{w}_k}\right|_q 
		\left.\frac{\partial \overline{f}_k}{\partial \overline{z}_j}\right|_p, \phantom{aaaa} 
		\left.\frac{\partial (\rho \circ f)}{\partial z_j}\right|_p 
		= \sum_{k=1}^n \left.\frac{\partial \rho}{\partial w_k}\right|_q 
		\left.\frac{\partial f_k}{\partial z_j}\right|_p,     $$
		$$ \left| \left.\frac{\partial (\rho \circ f)}{\partial z_j}\right|_p \right|^2  
		= \sum_{k,l=1}^n \left.\frac{\partial \rho}{\partial \overline{w}_k}\right|_q 
		\left.\frac{\partial \overline{f}_k}{\partial \overline{z}_j}\right|_p
		\left.\frac{\partial \rho}{\partial w_l}\right|_q 
		\left.\frac{\partial f_l}{\partial z_j}\right|_p,   $$
		and
		$$f_{*}\left( \sum_{j=1}^n \left. \frac{\partial (\rho \circ f)}{\partial \overline{z}_j} \right|_p \left.\frac{\partial}{\partial z_j}\right|_p \right) 
		= \sum_{j,k,l=1}^n \left.\frac{\partial f_l}{\partial z_j}\right|_p
		\left.\frac{\partial \rho}{\partial \overline{w}_k}\right|_q 
		\left.\frac{\partial \overline{f}_k}{\partial \overline{z}_j}\right|_p
		\left.\frac{\partial}{\partial w_l}\right|_q .    
		$$
		Hence, $f_{*}(L_{n,(\rho \circ f)}) \rho \equiv 1$ on $\partial \O_2$. 
		Let $f_{*}(L_{n,(\rho \circ f)}) = a L_{n,\rho} + \widetilde{L}$ for some $a \in \CC$ and $\widetilde{L} \in T^{1,0}_{f(p)}(\partial \O_2)$. Then
		$$ 1 = f_{*}(L_{n,(\rho \circ f)}) \rho = a L_{n,\rho} \rho + \widetilde{L} \rho = a ,$$
		and the claim is proved.
		
		Now, by Lemma \ref{lem alpha(L)} and the claim above,
		\begin{align*}
			\alpha_{(\rho \circ f)}(\overline{L})(p) 
			&= \Levi_{(\rho \circ f)}(L_{n,(\rho \circ f)}, L)(p)  \\
			&= \Levi_{\rho}(f_{*}(L_{n,(\rho \circ f)}), f_{*}(L))(f(p)) \\
			&= \Levi_{\rho}(L_{n,\rho}, f_{*}(L))(f(p))  + \Levi_{\rho}(\widetilde{L}, f_{*}(L))(f(p)) \\
			&= \alpha_{\rho}(\overline{f_{*}(L)})(f(p)).
		\end{align*}
		The explanation for the fourth equality is the following. 
		Since $f$ is a CR-map and $L \in \Null_{p}$, we have $f_{*}(L) \in \Null_{f(p)}$.
		Hence by Lemma \ref{orthgonal}, $\Levi_{\rho}(\widetilde{L}, f_{*}(L))(f(p)) = 0$.
		Therefore, the first equation (\ref{CR-invariance of D'Angelo 1 form}) is proved. \\

		For the second equation, we first claim that $ \overline{f_{*}(\nabla_{\overline{L}}L)} = \nabla_{f_{*}(L)} \overline{f_{*}(L)} $ on $\partial \O_2$.
		Let $f^{-1} = (f^{-1}_1, \cdots, f^{-1}_n)$ and  
		$$ L = \sum_{j=1}^{n} a_j \frac{\partial}{\partial z_j}. $$
		Then 
		\begin{align*}
		\nabla_{\overline{L}}L = \sum_{j,k=1}^n \overline{a}_k \frac{\partial a_j}{\partial \overline{z}_k} \frac{\partial}{\partial z_j}, \phantom{aaa}
		\overline{f_{*}(\nabla_{\overline{L}}L)} = \sum_{j,k,l=1}^n a_k \frac{\partial \overline{a}_j}{\partial z_k} \frac{\partial \overline{f}_l}{\partial \overline{z}_j} \frac{\partial}{\partial \overline{w}_l}, \\
		f_{*}(L) = \sum_{j,k=1}^{n} a_j \frac{\partial f_k}{\partial z_j} \frac{\partial}{\partial w_k}, \phantom{aaa} 
		\overline{f_{*}(L)} = \sum_{i,l=1}^{n} \overline{a}_i \frac{\partial \overline{f}_l}{\partial \overline{z}_i} \frac{\partial}{\partial \overline{w}_l}. 
		\end{align*}
		Therefore,
		\begin{align*}
			\nabla_{f_{*}(L)} \overline{f_{*}(L)} 
			&= \sum_{i,j,k,l,m=1}^{n} a_j \frac{\partial f_k}{\partial z_j} \frac{\partial \overline{a}_i}{\partial z_m} \frac{\partial f_m^{-1}}{\partial w_k} \frac{\partial \overline{f}_l}{\partial \overline{z}_i} \frac{\partial}{\partial \overline{w}_l} 
			= \sum_{i,j,l,m=1}^{n} a_j \frac{\partial \overline{a}_i}{\partial z_m} \delta^m_j \frac{\partial \overline{f}_l}{\partial \overline{z}_i} \frac{\partial}{\partial \overline{w}_l} \\
			&= \sum_{i,j,l=1}^{n} a_j \frac{\partial \overline{a}_i}{\partial z_j} \frac{\partial \overline{f}_l}{\partial \overline{z}_i} \frac{\partial}{\partial \overline{w}_l}
			= \overline{f_{*}(\nabla_{\overline{L}}L)},
		\end{align*}
		where $\delta^m_j$ is the Kronecker delta.
		
		Now by Lemma \ref{lem In terms of the D'Angelo one form 2},
		\begin{align*}
		i\left( d^c \alpha_{(f \circ \rho)} \right)(L,\overline{L})(p) 
		&= \left( d\alpha_{(f \circ \rho)} + id^c \alpha_{(f \circ \rho)} \right) (L, \overline{L})(p) \\
		&= 2 L\left( \alpha_{(f \circ \rho)}(\overline{L}) \right)(p) - 2 \alpha_{(f \circ \rho)}(\nabla_{L} \overline{L})(p) \\
		&= 2 L \left( \alpha_{\rho}(\overline{f_{*}(L)}) \circ f \right) (p) - 2 \alpha_{\rho}\left(  \overline{f_{*}(\nabla_{\overline{L}}L)} \right) (f(p))\\
		&= 2 f_{*}(L) \left( \alpha_{\rho}(\overline{f_{*}(L)}) \right) (f(p)) - 2 \alpha_{\rho}\left( \nabla_{f_{*}(L)} \overline{f_{*}(L)} \right) (f(p))\\
		&= id^c \alpha_{\rho} \left(f_{*}(L), \overline{f_{*}(L)} \right) (f(p)).
		\end{align*}
		
	\end{proof}

	\begin{proof}[\bf Proof of Theorem \ref{thm main CR-invariance of two indices}]

		We prove that $S(\O_1) = S(\O_2)$.
		If $S(\O_1)$ and $S(\O_2)$ are both infinity, then we are done.
		Hence, we may assume $S(\O_1) < \infty$. Then, since $\O_1$ admits a Stein neighborhood basis, $\O_1$ is a pseudoconvex domain. Moreover, CR-equivalance of $\partial \O_1$ and $\partial \O_2$ implies that $\O_2$ is a pseudoconvex domain. Therefore, we can apply Theorem~\ref{thm main DF,S D'Angelo 1-form} for both $\O_1$ and $\O_2$.
		
		Let $\rho$ be a defining function of $\O_2$.
		Let $\Sigma_1$ and $\Sigma_2$ be the sets of weakly pseudoconvex boundary points on $\partial \O_1$ and $\partial \O_2$, respectively. Fix a point $p \in \Sigma_1$.
		Suppose $f : \partial \O_1 \rightarrow \partial \O_2$ is a CR-diffeomorphism. Then since $f$ is a CR-diffeomorphism, $\Levi_{\rho}(f_{*}(L),f_{*}(L))(f(p)) = 0$ for all $p \in \Sigma_1$ and $L \in \Null_p$.
		Therefore, $f|_{\Sigma_1} : \Sigma_1 \rightarrow \Sigma_2$ and $d_pf|_{\Null_p} : \Null_p \rightarrow \Null_{f(p)}$ are injective and onto. 
		
		Now by Proposition \ref{prop CR-invariance of D'Angelo 1 form}, if for some $\eta_2 > 1$, 
		$$ \left( \frac{\eta_2}{\eta_2 - 1}(\pi_{1,0} \alpha_{\rho} \wedge \pi_{0,1} \alpha_{\rho}) - \frac{i}{2} d^c \alpha_{\rho} \right) (f_{*}(L), \overline{f_{*}(L)}) \le 0, $$
		for all $L \in \Null_p$, then 
		$$ \left( \frac{\eta_2}{\eta_2 - 1}(\pi_{1,0} \alpha_{(\rho \circ f)} \wedge \pi_{0,1} \alpha_{(\rho \circ f)}) - \frac{i}{2} d^c \alpha_{(\rho \circ f)} \right) (L, \overline{L}) \le 0 .$$
		for all $L \in \Null_p$.
		Therefore, this implies $S(\O_1) \le S(\O_2)$ by using Theorem \ref{thm main DF,S D'Angelo 1-form} and the equation (\ref{4.5}). The same argument for $f^{-1}$ gives $S(\O_1) \ge S(\O_2)$, and this completes the proof. The statement $DF(\O_1) = DF(\O_2)$ follows from the same argument as in the Steinness index case.

	\end{proof}

\vspace{5mm}

	
	\textit{\bf Acknowledgements}: 
	This work is completed while the author visited the University of Wuppertal in May, 2019. The author would like to thank Professor N. Shcherbina, T. Pawlaschyk, T. Harz and H. Herrmann not only for warm hospitality but also for many fruitful conversations.
	He is also grateful to M. Adachi for his careful reading and comments to imporve this paper.  
	

	\vspace{5mm}

	\bibliographystyle{amsplain}

\end{document}